\documentclass[12pt]{article}

\usepackage{upgreek}

\usepackage{hyperref}

\usepackage[normalem]{ulem}
\usepackage{amsthm}
\usepackage{amsmath}
\usepackage{amssymb}
\usepackage{mathrsfs}
\usepackage{mathtools}
\usepackage{graphicx}
\usepackage{tikz}
\usetikzlibrary{arrows}
\usepackage{enumerate}
\usepackage{comment}
\usepackage{caption}
\usepackage{subcaption}

\usepackage{fullpage}

\usepackage{appendix}

\newtheorem{theorem}{Theorem}
\newtheorem{conjecture}[theorem]{Conjecture}
\newtheorem{lemma}[theorem]{Lemma}
\newtheorem{proposition}[theorem]{Proposition}
\newtheorem{corollary}[theorem]{Corollary}

\theoremstyle{definition}
 
\theoremstyle{remark}

\theoremstyle{remark}

\theoremstyle{observation}
\newtheorem{observation}[theorem]{Observation}

\numberwithin{theorem}{section}

\providecommand{\Z}{}
\providecommand{\N}{}

\renewcommand{\Z}{\mathbb{Z}}
\renewcommand{\N}{{\mathbb N}}

\newcommand{\E}{\textsf{\upshape E}}
\newcommand{\prob}{\textsf{\upshape Pr}}
\newcommand{\Var}{\textsf{\upshape Var}}

\newcommand{\wt}[1]{\widetilde{#1}}

\begin{document}

\title{Burning Random Trees}

\author{
Luc Devroye\thanks{School of Computer Science, McGill University, Montreal, Canada} 
\and
Austin Eide\thanks{Department of Mathematics, Toronto Metropolitan University, Toronto, Canada}
\and
Pawe\l{} Pra\l{}at\thanks{Department of Mathematics, Toronto Metropolitan University, Toronto, Canada}
}

\maketitle

\begin{abstract}
Let $\mathcal{T}$ be a Galton-Watson tree with a given offspring distribution $\xi$, where $\xi$ is a $\Z_{\geq 0}$-valued random variable with $\E[\xi] = 1$ and $0 < \sigma^{2}:=\Var[\xi] < \infty$. For $n \geq 1$, let $T_{n}$ be the tree $\mathcal{T}$ conditioned to have $n$ vertices. In this paper we investigate $b(T_n)$, the burning number of $T_n$. Our main result shows that asymptotically almost surely $b(T_n)$ is of the order of $n^{1/3}$.
\end{abstract}

\section{Introduction}

Graph burning is a discrete-time process that models influence spreading in a network. Vertices are in one of two states: either \emph{burning} or \emph{unburned}. In each round, a burning vertex causes all of its neighbors to burn and a new \emph{fire source} is chosen: a previously unburned vertex whose state is changed to burning. The updates repeat until all vertices are burning. The \emph{burning number} of a graph $G$, denoted $b(G)$, is the minimum number of rounds required to burn all of the vertices of $G$.  

Graph burning first appeared in print in a paper of Alon~\cite{alon1992transmitting}, motivated by a question of Brandenburg and Scott at Intel, and was formulated as a transmission problem involving a set of processors. It was then independently studied by Bonato, Janssen, and Roshanbin~\cite{bonato2014burning,bonato2016burn} who, in addition to introducing the name \emph{graph burning}, gave bounds and characterized the burning number for various graph classes. The problem has since received wide attention (e.g.~\cite{bastide2021improved,bessy2018bounds,land2016upper,mitsche2017burning,mitsche2018burning,norin2024burning}), with particular focus the so-called Burning Number Conjecture that each connected graph on $n$ vertices requires at most $\lceil \sqrt{n} \rceil$ turns to burn. This conjecture is best possible, as $b(P_n) = \left\lceil \sqrt{n} \right\rceil$ for a path on $n$ vertices.

Clearly, $b(G) \le b(T)$ for every spanning tree $T$ of $G$. Hence, the Burning Number Conjecture can be stated as follows:
\begin{conjecture}
For any tree $T$ on $n$ vertices, $b(T) \le \left\lceil \sqrt{n} \right\rceil$.
\end{conjecture}
Although the conjecture feels obvious, it has resisted attempts at its resolution. It is easy to couple the process on any tree $T$ on $n$ vertices ($n-1$ edges) with the process on $C_{2(n-1)}$, a cycle on twice as many edges. (See Figure~\ref{fig:tree_and_cycle}.) This yields
$$
b(T) \le \left\lceil \sqrt{2(n-1)} \right\rceil = \sqrt{2n} + O(1) \qquad (\sqrt{2} \approx 1.41421).
$$

\begin{figure}
\centering
\begin{subfigure}{.5\linewidth}
  \centering
  \includegraphics[width=.7\linewidth]{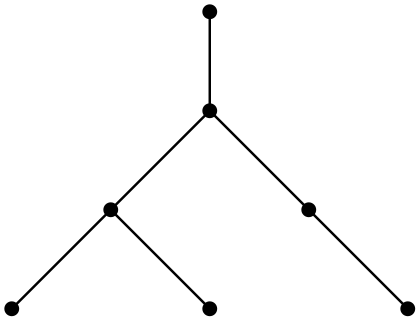}
  \caption{}
\end{subfigure}%
\begin{subfigure}{.5\linewidth}
  \centering
  \includegraphics[width=.7\linewidth]{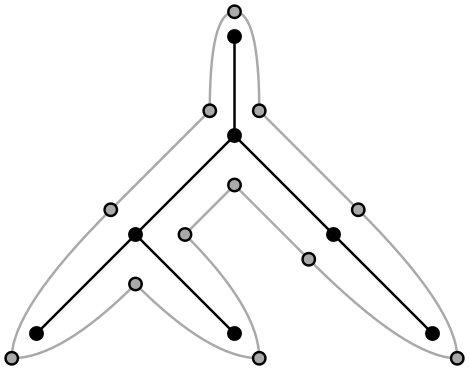}
  \caption{}
\end{subfigure}
\caption{A tree $T$ (a) and a cycle corresponding to a depth-first search of $T$ (b). Any burning sequence for the cycle projects to a burning sequence for the tree.}
\label{fig:tree_and_cycle}
\end{figure}

\noindent
In~\cite{bessy2018bounds}, Bessy, et.\ al.\ proved that
$$
b(T) \le \sqrt{ \frac {12}{7} n} + 3 \qquad (\sqrt{12/7} \approx 1.30931).
$$
This bound was consecutively improved by Land and Lu in~\cite{land2016upper} to
$$
b(T) \le \left\lceil \frac { \sqrt{24n+33} - 3 }{4} \right\rceil = \sqrt{ \frac {3}{2} \, n } + O(1) \qquad (\sqrt{3/2} \approx 1.22475).
$$
Currently, the best upper bound, due to Bastide, et.\ al.\ in~\cite{bastide2021improved}, is
$$
b(T) \le \left\lceil \sqrt{ \frac {4}{3} n } \right\rceil + 1 \qquad (\sqrt{4/3} \approx 1.15470).
$$
Norin and Turcotte~\cite{norin2024burning} proved arguably the strongest result in this direction: 
$$
b(T) \le (1+o(1)) \sqrt{n}.
$$
That is, the conjecture holds asymptotically.

\bigskip

We intend to investigate the burning number of random trees. Let $\mathcal{T}$ be a Galton---Watson tree with a given offspring distribution $\xi$, where $\xi$ is a $\Z_{\geq 0}$-valued random variable with 
\begin{equation}\label{eq:distribution}
\E[\xi] = 1, \qquad \text{ and } \qquad 0 < \sigma^{2}:=\Var[\xi] < \infty. 
\end{equation}
In other words, the Galton---Watson tree is critical, of finite variance, and satisfies $\prob(\xi = 1) < 1$. In particular, it implies that $0 < \prob(\xi = 0) < 1$.

For $n \geq 1$, let $T_{n}$ be the tree $\mathcal{T}$ conditioned to have $n$ vertices. (Implicitly, we assume $\prob(|\mathcal{T}| = n) > 0$. This requires a straightforward assumption on $n$ and $\xi$ which we make explicit later; see Section \ref{sec:upper_bound}.) The resulting random tree is an instance of the family of \textit{simply generated trees} introduced by Meir and Moon~\cite{meir1978altitude}. This family contains many combinatorially interesting random trees such as uniformly chosen random plane trees, random unordered labelled trees (known as Cayley trees), and random $d$-ary trees. For more examples, see, Aldous~\cite{aldous1991continuum} and Devroye~\cite{devroye1998branching}. For more on the relationship between conditioned Galton---Watson trees and simply generated trees, see Janson~\cite[Section 4]{janson2012simply}.

\medskip

Our main result shows that, with high probability, $b(T_n)$ is of the order of $n^{1/3}$. 

\begin{theorem}\label{thm:main_theorem}
    Let $T_{n}$ be a conditioned Galton---Watson tree of order $n$, subject to~(\ref{eq:distribution}). For any $\epsilon = \epsilon(n)$ tending to $0$ as $n \to \infty$, we have
    $$
        \prob \Big( ( \epsilon n )^{1/3} \le b(T_{n}) \le (n/\epsilon)^{1/3} \Big) = 1 - O(\epsilon).
    $$
\end{theorem}

The paper is structured as follows. First, we make a simple observation that the burning number can be reduced to the problem of covering vertices of the graph with balls, a slightly easier problem; see Section~\ref{sec:covering}. Section~\ref{sec:lower_bound} is devoted to the lower bound for the burning number whereas the upper bound is provided in Section~\ref{sec:upper_bound}. We finish the paper with a few natural questions; see Section~\ref{sec:future}.

\section{Covering a Graph with Balls}\label{sec:covering}

In this section, we show a simple but convenient observation that reduces the burning number to the problem of covering the graph's vertices with balls. Let $G = (V, E)$ be any graph. For any $r \in \N_{0}$ and vertex $v \in V$, we denote by $B_r(v)$ the ball of radius $r$ centered at $v$, that is, $B_r(v) = \{ u \in V : d(u,v) \le r \}$, where $d(u,v)$ denotes the distance between $u$ and $v$.

First, note that since the burning process is deterministic, a fire source $v$ makes all vertices in $B_t(v)$ burn after $t$ rounds but only those vertices are affected. As a result, the burning number can be reformulated as follows: 
$$
b(G) = \min \left\{ k \in \N : \exists v_0, v_1, \ldots, v_{k-1} \in V \text{ such that } \bigcup_{r=0}^{k-1} B_r(v_r) = V \right\}.
$$
Dealing with balls of different radii is inconvenient so we will simplify the problem slightly by considering balls of the same radii. Let $\hat{b}(G)$ be the counterpart of $b(G)$ for this auxiliary covering problem, that is,
$$
\hat{b}(G) = \min \left\{ k \in \N : \exists v_1, v_2, \ldots, v_{k} \in V \text{ such that } \bigcup_{r=1}^{k} B_k(v_r) = V \right\}.
$$

Covering with $k$ balls of increasing radii (in particular, all of them of radii at most $k-1$) is not easier than covering with $k$ balls of radius $k$. Hence, $\hat{b}(G) \le b(G)$. On the other hand, covering with $2k$ balls of increasing radii (in particular, $k$ of them of radii at least $k$) is not more difficult than covering with $k$ balls of radius $k$ implying that $b(G) \le 2 \hat{b}(G)$. We conclude that $\hat{b}(G)$ and $b(G)$ are of the same order:

\begin{observation}\label{obs:b_hat}
For any graph $G=(V,E)$, 
$$
\hat{b}(G) \le b(G) \le 2 \hat{b}(G).
$$
\end{observation}

\noindent
In particular, we may prove the bounds in our main result (Theorem~\ref{thm:main_theorem}) for $\hat{b}(G)$ instead of $b(G)$, which will be slightly easier. 

\section{Lower bound}\label{sec:lower_bound}

For an arbitrary tree $\tau$ and $i \in \N$, let 
$$
P_{i}(\tau) := \left| \Big\{ \{v,w\}: v,w \in V(\tau),\,d(v,w) = i \Big\} \right|.
$$
In other words, $P_{i}(\tau)$ counts the number of unordered pairs of vertices which are distance $i$ apart in $\tau.$ We have the following result from~\cite{devroye2011distances} that upper bounds the expected number of such pairs in the random tree $T_n$: 

\begin{theorem}[{\cite[Theorem 1.3]{devroye2011distances}}]\label{distance_theorem}
There exists a constant $c > 0$, dependent on the distribution of $\xi$, such that for all $n \in \N$ and $i \in \N$, $\E[P_{i}(T_{n})] \leq cni$.
\end{theorem}

The essence of the lower bound on $b(T_{n})$ is the following: Theorem~\ref{distance_theorem} suggests that a typical ball of radius $j$ in $T_{n}$ contains $O(j^{2})$ vertices; thus, covering $T_{n}$ with $j$ balls of radius at most $j$ requires $j$ to satisfy $j^{3} = \Omega(n)$---that is, $j = \Omega(n^{1/3})$. We make this argument rigorous in the proof of Proposition~\ref{burning_lower_bound}.

\begin{proposition}\label{burning_lower_bound}
 Let $k = k(n) = (\epsilon n)^{1/3}$, where $\epsilon = \epsilon(n) \to 0$ as $n \to \infty$. With probability $1-O(\epsilon)$, $\hat{b}(T_n) \ge k$, that is, there is no partition of the vertices of $T_{n}$ into $k$ disjoint sets $U_{1},U_{2},\dots,U_{k}$ such that \normalfont{$\text{diam}(U_{j}) \leq 2k$} for all $1 \leq j \leq k$.
\end{proposition}

Since $b(T_n) \ge \hat{b}(T_n)$ (Observation~\ref{obs:b_hat}), Proposition~\ref{burning_lower_bound} implies the corresponding lower bound in Theorem~\ref{thm:main_theorem}.

\begin{proof}
Let $Q_{j}(T_{n}) = \sum_{i=1}^{j}P_{i}(T_{n})$, that is, $Q_{j}(T_{n})$ counts pairs of vertices in $T_{n}$ which are at most distance $j$ apart. From Theorem \ref{distance_theorem}, for all $n,j \in \N$, 
$$
\E[Q_{j}(T_{n})] = \sum_{i=1}^{j}\E[P_{i}(T_{n})]  \leq cn \sum_{i=1}^{j}i \leq cnj^{2}.
$$

For a contradiction, suppose there is a partition $U_{1},U_{2},\dots,U_{k}$ of the vertices of $T_{n}$ as described in the statement of the proposition. Since every pair of vertices in a given $U_{j}$ is at most distance $2k$ apart, we must have 
$$
Q_{2k}(T_{n}) \geq \sum_{j=1}^{k}\binom{|U_{j}|}{2} = \sum_{j=1}^{k}\left( \frac{|U_{j}|^{2}}{2} - \frac{|U_{j}|}{2} \right) = \frac{1}{2}\sum_{j=1}^{k}|U_{j}|^{2} - \frac{n}{2}.
$$
By Jensen's inequality, we get
$$
\frac{1}{k}\sum_{j=1}^{k}|U_{j}|^{2} \geq \left(\frac{1}{k}\sum_{j=1}^{k}|U_{j}|\right)^{2} = \left( \frac{n}{k} \right)^{2}
$$ 
and therefore
$$
Q_{2k}(T_{n}) \geq \frac{n^{2}}{2k} - \frac{n}{2} = \frac{n^{2}}{2k}\left(1 - O(k/n)\right) = \frac{n^{2}}{2k}\left(1 - o(n^{-2/3})\right).
$$

On the other hand, $\E[Q_{2k}(T_{n})] \leq 4cnk^{2}$. Thus, by Markov's inequality, 
\begin{eqnarray*}
\prob \left(Q_{2k}(T_{n}) \geq \frac{n^{2}}{2k} - \frac{n}{2}\right) &\leq& \frac{\E[Q_{2k}(T_{n})]}{\frac{n^{2}}{2k}} \left(1 + o(n^{-2/3})\right) \\
&=& O\left(\frac{k^{3}}{n}\right)
= O(\epsilon).
\end{eqnarray*}
It follows that a partition of $V(T_{n})$ with the stated properties exists with probability $O(\epsilon)$, which finishes the proof of the proposition.
\end{proof}

\section{Upper bound}\label{sec:upper_bound}

Here we prove the upper bound of Theorem \ref{thm:main_theorem}. To do this, we first devise a deterministic strategy which covers any rooted tree $\tau$ with balls of radius $2k$. The balls are centered at the vertices of a particular subset $C \subset V(\tau)$; in the Galton---Watson setting, we are able to show that $|C|$ is less than or equal to $2k$ a.a.s.\ as long as $k$ is at least of the order $n^{1/3}$, which yields the desired upper bound.

For any rooted tree $\tau$ with root $r$, for $i \in \N_{0}$ we write $\ell_{i}(\tau):=\{v \in V(\tau)\,:\,d_{\tau}(r,v) = i\}$ for the set of vertices at depth $i$. Let $h(\tau) \in \N_{0} \cup \{\infty\}$ be the height of $\tau$, that is, $h(\tau) := \sup\{i\,:\,\ell_{i}(\tau) \neq \emptyset \}$. For any $v \in V(\tau)$, let $\tau_{v}$ the (full) sub-tree of $\tau$ rooted at $v$. For $k \in \N$, and $j \in \{0,1,\dots,k-1\}$, let
$$
\mathcal{C}_{k}^{j}(\tau) := \bigcup_{i=0}^{\infty} \Big\{v \in \ell_{ik +  j}(\tau)\,:\,h(\tau_{v}) \geq k \Big\}.
$$
So $\mathcal{C}_{k}^{j}(\tau)$ consists of all vertices whose depth is $j$ modulo $k$ with subtrees of height at least $k$.

We first show that placing balls of radius $2k$ at the root and at each vertex in $\mathcal{C}_{k}^{j}(\tau)$ covers the vertices of $\tau$.

\begin{lemma}\label{covering_lemma}
    Let $\tau$ be a tree rooted at $r$. For any $k \in \N$ and $j \in \{ 0,1,\dots,k-1 \}$ we have
    $$
    V(\tau) = \bigcup_{v \in \mathcal{C}_{k}^{j}(\tau) \cup \{r\}}B_{2k}(v).
    $$
\end{lemma}

\begin{proof}
    Fix $k \in \N$ and $j \in \{ 0,1,\dots,k-1 \}$. Let $v \in V(\tau)$ and let $i$ be the smallest non-negative integer such that $d(r,v) < ik +j$ and let $a$ be the unique ancestor of $v$ in $\ell_{(i-2)k +j}(\tau)$, or $a = r$ if $i \in \{0,1\}$. Then $a$ is either the root $r$, or $d(r,a)\equiv j\pmod k$ and $h(\tau_{a}) \geq k$. In either case, $a \in \mathcal{C}_{k}^{j}(\tau) \cup \{r\}$. Since $d(a,v) \leq 2k$, we have $v \in B_{2k}(a).$ This finishes the proof of the lemma. 
\end{proof}

Lemma \ref{covering_lemma} provides a scheme to cover a general rooted tree $\tau$ with balls of radius $2k$. Observe that for any $j$, $|\mathcal{C}_{k}^{j}(\tau)| \leq 2k-1$ implies that $\hat{b}(\tau) \leq 2k$, and hence $b(\tau)\leq 4k$ by Observation~\ref{obs:b_hat}. In particular, we conclude the following: 
\begin{equation}
\text{if } \min_{j}|\mathcal{C}_{k}^{j}(\tau)| \leq 2k-1, \text{ then }b(\tau) \leq 4k. \label{eq:upper_bound}
\end{equation}

\bigskip

Our next lemma estimates the probability that a vertex selected uniformly at random from the random tree $T_{n}$ has height at least $k$.

\begin{lemma}\label{random_vertex}
Consider the random tree $\tau = T_n$ on $n$ vertices. Then there exists a constant $c > 0$, dependent on the distribution of $\xi$, such that the following property holds. Let $u$ be a vertex selected uniformly at random from $V(\tau)$, and let its subtree be denoted by $\tau_u$. Then, for all $k \ge 0$,
$$
\prob \Big( h(\tau_u) \ge k \Big) \le  c\left(\frac{1}{k} + \frac{1}{\sqrt{n}} \right).
$$
\end{lemma}

To prove Lemma~\ref{random_vertex} we require some standard tools for conditioned Galton---Watson trees, which we introduce now. Given a tree $\tau$ with $s$ vertices, let $v_{1},v_{2},\dots,v_{s}$ be the vertices of $\tau$ in depth-first search ({\sc dfs}) order. We write $d_{i}$ for the number of children of $v_{i}$ and refer to $(d_{1},d_{2},\dots,d_{s})$ as the \textit{preorder degree sequence} of $\tau$. The preorder degree sequence gives rise to a representation of $\tau$ as a lattice path $(j, y_j)_{j=0}^s$ started from $(0,y_0)=(0,0)$ with $s$ steps, where the $j$th step is given by $y_j = y_{j-1} + (d_{j}-1)$. See Figure~\ref{fig:tree_and_lattice_path} for an example. 

\begin{figure}
\centering
\begin{subfigure}{.4\linewidth}
  \centering
  \includegraphics[width=.7\linewidth]{tree.png}
  \caption{}
\end{subfigure}%
\begin{subfigure}{.6\linewidth}
  \centering
  \includegraphics[width=.7\linewidth]{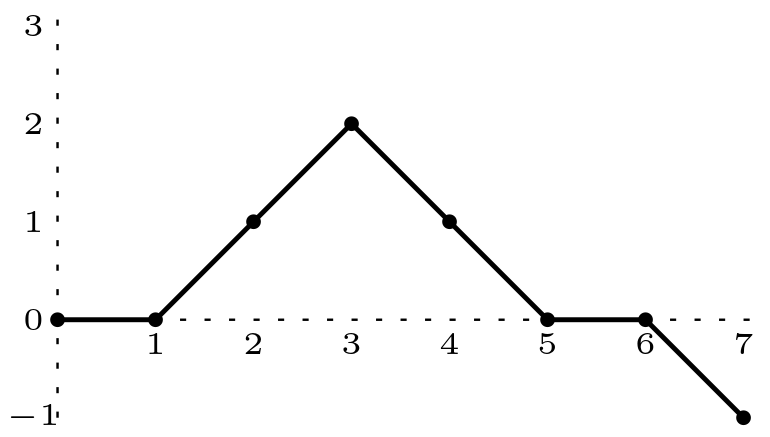}
  \caption{}
\end{subfigure}
\caption{A rooted tree $\tau$ with preorder degree sequence $(1,2,2,0,0,1,0)$ (a) and its lattice path representation (b).}
\label{fig:tree_and_lattice_path}
\end{figure}

Note that the lattice path corresponding to a tree with $s$ vertices always ends at the point $(s,y_s)=(s, -1)$ and has a height strictly greater than $-1$ before then, evidenced by the fact that the height of the path at step $j$ is one less than the number of vertices in the ``queue" of the {\sc dfs} of the tree after $j$ vertices have been explored. Since vertex $v_{s}$ is always a leaf in the dfs order on $\tau$, the point $(s-1, y_{s-1})$ must equal $(0,0)$, meaning that the lattice path representation of $\tau$ can be identified with a \textit{\L{}ukasiewicz path} of length $s-1$. As the next lemma summarizes, there is a bijective correspondence between ordered trees with $s$ vertices and lattice paths of this type. 

\begin{lemma}[{{\cite[Lemma 15.2]{janson2012simply}}}]\label{degree_sequence}
    A sequence $(d_{1},d_{2},\dots,d_{s}) \in \N_{0}^{s}$ is the preorder degree sequence of a tree if and only if
        $$
            \sum_{i=1}^{j}d_{i} \geq j \quad \forall\,j \in \{1,2,\dots,s-1\}
        $$
    and
        $$
            \sum_{i=1}^{s}d_{j} = s-1.
        $$
\end{lemma}

We also have the following useful property. 

\begin{lemma}[{{\cite[Corollary 15.4]{janson2012simply}}}]\label{cyclic_shift}
    If $(d_{1},d_{2},\dots,d_{s}) \in \N_{0}^{s}$ satisfies $\sum_{i=1}^{s}d_{i} = s-1$, then precisely one of the cyclic permutations of $(d_{1},d_{2},\dots,d_{s})$ is the preorder degree sequence of a tree.
\end{lemma}

Let $(\xi_{1}^{(s)}, \xi_{2}^{(s)}, \dots, \xi_{s}^{(s)})$ be the preorder degree sequence of a Galton---Watson tree $T_{s}$ with offspring distribution $\xi$ conditioned to have $s$ vertices, and let $(\wt{\xi}_{1}^{(s)}, \wt{\xi}_{2}^{(s)}, \dots, \wt{\xi}_{s}^{(s)})$ be a uniformly random cyclic permutation of $(\xi_{1}^{(s)}, \xi_{2}^{(s)}, \dots, \xi_{s}^{(s)})$. Let $\xi_{1},\xi_{2},\dots$ be a sequence of i.i.d.\ copies of $\xi$, and define, for any $j \geq 1$, $S_{j} = \sum_{i=1}^{j}\xi_{i}$. Lemmas~\ref{degree_sequence} and~\ref{cyclic_shift} yield the following corollary:

\begin{corollary}\label{uniform_cyclic_shift}
    The sequence $(\wt{\xi}_{1}^{(s)}, \wt{\xi}_{2}^{(s)}, \dots, \wt{\xi}_{s}^{(s)})$ has the same distribution as the sequence $(\xi_{1},\xi_{2},\dots,\xi_{s})$ conditioned on $S_{s} = s-1$.
\end{corollary}

Define the \textit{span} of $\xi$ as
$$
h = \gcd\{i \geq 1\,:\,\prob(\xi = i) > 0\}.
$$
We will use the following local limit theorems (see \cite[Lemma~4.1 and~(4.3)]{janson2016asymptotic} and the sources referenced therein).

\begin{lemma}\label{llt}
    Suppose $\xi$ satisfies (\ref{eq:distribution}) and has span $h$. Then, as $s \to \infty$, uniformly for $m \equiv 0\pmod h$, 
    $$
        \prob(S_{s} = m) = \frac{h}{\sqrt{2\pi\sigma^{2}s}}\left(e^{-(m-s)^{2}/2s\sigma^{2}} + o(1) \right).
    $$
    If $\mathcal{T}$ is the Galton---Watson tree with offspring distribution $\xi$, then for $s \equiv 1\pmod h$, as $s \to \infty$,
    $$
        \prob(|\mathcal{T}| = s) = \frac{h}{\sqrt{2\pi\sigma^{2}s^{3}}}(1+o(1)).
    $$
\end{lemma}

We are now ready to prove Lemma~\ref{random_vertex}.

\begin{proof}[Proof of Lemma~\ref{random_vertex}] 
Throughout the proof, we will use $c_{1}, c_{2}, \dots$ for non-explicit positive constants which do not depend on $n$ (but may depend on the distribution of $\xi$). Implicitly, we will assume throughout that $n \equiv 1\pmod h$, where $h$ is the span of $\xi$, so that $\prob(S_{n} = n-1) > 0.$

We identify the random vertex $u$ with a random index of the {\sc dfs} order on $T_{n}$. Consider $(\xi_{u}^{(n)}, \xi_{u+1}^{(n)}, \dots,\xi_{n}^{(n)}, \xi_{1}^{(n)}, \dots, \xi_{u-1}^{(n)}).$ It is clear that this sequence has the same distribution as $(\wt{\xi}_{1}^{(n)}, \wt{\xi}_{2}^{(n)}, \dots, \wt{\xi}_{n}^{(n)})$, which, in turn, has the same distribution as $(\xi_{1},\xi_{2},\dots,\xi_{n})$ conditioned on $S_{n} = n-1$ by Corollary~\ref{uniform_cyclic_shift}.

For $k \in \N$ and $s \geq k$, let $\mathfrak{T}_{s}^{k}$ be the set of ordered trees with $s$ vertices and height at least $k$. For a sequence $(d_{1},d_{2},\dots,d_{s}) \in \N_{0}^{s}$, we write $(d_{1},d_{2},\dots,d_{s}) \in \mathfrak{T}_{s}^{k}$ if $(d_{1},d_{2},\dots,d_{s})$ is the preorder degree sequence of a tree in $\mathfrak{T}_{s}^{k}$. Note that for any $(d_{1},d_{2},\dots,d_{s}) \in \mathfrak{T}_{s}^{k}$, we have $\sum_{i=1}^{s}d_{i} = s-1$.
Then,
\begin{eqnarray}
    \prob(h(\tau_{u}) \geq k) &=& \sum_{s = k}^{n}\prob((\wt{\xi}_{1}^{(n)}, \wt{\xi}_{2}^{(n)},\dots,\wt{\xi}_{s}^{(n)}) \in \mathfrak{T}_{s}^{k})\nonumber\\
    &=&\sum_{s=k}^{n}\prob((\xi_{1}, \xi_{2},\dots,\xi_{s}) \in \mathfrak{T}_{s}^{k}\,|\,S_{n} = n-1)\nonumber\\
    &=&\sum_{s=k}^{n}\frac{\prob((\xi_{1}, \xi_{2},\dots,\xi_{s}) \in \mathfrak{T}_{s}^{k},\,S_{n} = n-1)}{\prob(S_{n} = n-1)}\nonumber\\
    &=&\sum_{s=k}^{n}\frac{\prob((\xi_{1}, \xi_{2},\dots,\xi_{s}) \in \mathfrak{T}_{s}^{k}) \cdot \prob\left(\sum_{i=s+1}^{n}\xi_{i} = n-s \right)}{\prob(S_{n} = n-1)}\nonumber\\
    &=&\sum_{s=k}^{n}\frac{\prob((\xi_{1}, \xi_{2},\dots,\xi_{s}) \in \mathfrak{T}_{s}^{k}) \cdot \prob(S_{n-s} = n-s )}{\prob(S_{n} = n-1)}.\label{eq:sum1}
\end{eqnarray}

By Lemma~\ref{llt}, there is a constant $c_{1} > 0$ such that for $s \in \{1, 2, \dots, \lfloor n/2 \rfloor\}$, we have
$$
\frac{\prob(S_{n-s} = n-s)}{\prob(S_{n} = n-1)} \leq c_{1}.
$$
Thus,
\begin{eqnarray*}
\sum_{s=k}^{\lfloor n/2 \rfloor}\frac{\prob((\xi_{1}, \xi_{2},\dots,\xi_{s}) \in \mathfrak{T}_{s}^{k}) \cdot \prob(S_{n-s} = n-s )}{\prob(S_{n} = n-1)} &\leq& c_{1}\sum_{s=k}^{\lfloor n/2 \rfloor}\prob((\xi_{1}, \xi_{2},\dots,\xi_{s}) \in \mathfrak{T}_{s}^{k})\\
&\leq& c_{1}\sum_{s=k}^{\infty}\prob((\xi_{1}, \xi_{2},\dots,\xi_{s}) \in \mathfrak{T}_{s}^{k})\\
&=& c_{1} \, \prob(h(\mathcal{T}) \geq k),
\end{eqnarray*}
where, recall, $\mathcal{T}$ is the unconditioned Galton---Watson tree with offspring distribution $\xi$. By Kolmogorov's Theorem \cite[Theorem 12.7]{lyons2017probability}, there is a constant $c_{2} > 0$ so that for any $k \geq 1$, we have $\prob(h(\mathcal{T}) \geq k) \leq c_{2} / k$, and thus $c_{1} \, \prob(h(\mathcal{T}) \geq k) \leq c_{3} / k$ for some constant $c_{3} > 0$. This bounds the partial sum of~(\ref{eq:sum1}) where $k \leq s \leq n/2$. 

For the other partial sum, we first observe that
$$
\prob((\xi_{1}, \xi_{2},\dots,\xi_{s}) \in \mathfrak{T}_{s}^{k}) \leq \prob(|\mathcal{T}| = s).
$$
By Lemma~\ref{llt}, there is a constant $c_{4} > 0$ so that
$$
\frac{\prob(|\mathcal{T}| = s)}{\prob(S_{n} = n-1)} \leq \frac{c_{4}}{n}
$$
for all $n \geq 1$ and $\lceil n/2 \rceil \le s \leq n$. Using Lemma~\ref{llt} again, we get that for all $j \geq 1$
$$
\prob(S_{j} = j) \leq \frac{c_{5}}{\sqrt{j}}
$$
for some constant $c_{5} > 0$. It follows that
\begin{eqnarray*}
\sum_{s = \lceil n/2 \rceil}^{n}\prob(S_{n-s} = n-s) &=& \sum_{j=0}^{n-\lceil n/2 \rceil}\prob(S_{j} = j)\\
&\leq& 1 + c_{5}\sum_{j=1}^{n-\lceil n/2 \rceil}\frac{1}{\sqrt{j}}\\
&\leq&c_{6}\sqrt{n}
\end{eqnarray*}
for some constant $c_{6} > 0$, where the bound in the last line follows from a straightforward comparison of the sum with an integral. In all, we get that the partial sum of~(\ref{eq:sum1}) with $n/2 \le s \leq n$ is at most $c_{7} / \sqrt{n}$ (for some constant $c_7 > 0$) for all $n$.

Finally, combining the two bounds, we conclude that~(\ref{eq:sum1}) is upper bounded by $c_{8}\left(\frac{1}{k} + \frac{1}{\sqrt{n}} \right)$ for some constant $c_8>0$, as desired. This finishes the proof of the lemma.
\end{proof}

We now have all the ingredients to finalize the upper bound. The next theorem, as discussed earlier (see~(\ref{eq:upper_bound})), implies that $b(\tau) \leq 4k = O( (n/\epsilon)^{1/3} )$ with probability $1 - O(\epsilon)$.

\begin{theorem}\label{thm:upper_bound}
    Let $\epsilon = \epsilon(n) \to 0$ as $n \to \infty$, and let $k = \lfloor \left(\frac{n}{\epsilon} \right)^{1/3} \rfloor$. Then, with probability $1 - O(\epsilon)$, we have
    $$
    \min_{j \in \{0, 1, \dots,k-1\} } |\mathcal{C}_{k}^{j}(T_{n})| \leq 2k-1.
    $$
\end{theorem}

\begin{proof}
Let $X_{j} = |\mathcal{C}_{k}^{j}(T_{n})|$ for $j \in \{0,1,\dots,k-1\}$, and let $Y$ be the number of vertices $v$ in $\tau = T_{n}$ such that $h(\tau_{v}) \geq k$. Clearly, we have the identity $Y = \sum_{j=0}^{k-1}X_{i}.$

Now, observe that
$$
\min_{j \in \{0, 1, \dots, k-1\} } X_{j} \leq \frac{1}{k} \sum_{j=0}^{k-1}X_{j} = \frac{Y}{k}.
$$
By Lemma~\ref{random_vertex}, $\E[Y] \leq c\left(\frac{n}{k} + \sqrt{n} \right)$. Therefore, by Markov's inequality,
$$
\prob\left(\frac{Y}{k} > 2k-1 \right) \leq  \frac {\E[Y]}{(2k-1)k} \le \frac{c}{2k-1}\left(\frac{n}{k^{2}} + \frac{\sqrt{n}}{k} \right) = O\left(\frac{n}{k^{3}} \right) = O(\epsilon).
$$
This completes the proof of the theorem.
\end{proof}

\section{Future Directions}\label{sec:future}

In this paper we showed that asymptotically almost surely (a.a.s.) $b(T_n)$ is close to $n^{1/3}$, that is, a.a.s.\ $\epsilon n^{1/3} \le b(T_n) \le n^{1/3} / \epsilon$, provided that $\epsilon=\epsilon(n) \to 0$ as $n \to \infty$. A more detailed analysis gives that $c_{1}\left(\frac{n}{\sigma^{2}}\right)^{1/3} \leq \E[b(T_{n})] \leq c_{2}\left(\frac{n}{\sigma^{2}}\right)^{1/3}$, where $c_{1}$ and $c_{2}$ are universal constants, i.e., not dependent on the offspring distribution $\xi$. One could ask under what conditions there exists a universal constant $c$ such that $\E[b(T_{n})] = (c+o(1))\left(\frac{n}{\sigma^{2}}\right)^{1/3}$.  A limit theory for $b(T_{n})/n^{1/3}$ would also be of interest.

Achieving this improvement with our techniques would require significant refinements of Theorem \ref{distance_theorem} and Lemma \ref{random_vertex}. For instance, Theorem \ref{distance_theorem} implies that the expected number of pairs of vertices within distance $k$ of one another in $T_{n}$ is $O(nk^{2}),$ with the implied constant dependent on $\xi$. If instead we had that a.a.s.\ the number of such pairs is $O(nk^{2})$, the proof of Proposition \ref{burning_lower_bound} would (with minimal modification) give that $b(T_{n}) \geq \Omega(n^{1/3})$. Similarly, if we knew that a.a.s.\ the number of full subtrees of height at least $k$ was at most $O(n/k)$, then the proof of Theorem \ref{thm:upper_bound} would imply that a.a.s.\ $b(T_{n}) = O(n^{1/3})$. At present, Lemma~\ref{random_vertex} only bounds the expected number of such subtrees. Refinements of this type would be interesting results in their own right.

\section{Acknowledgement}

Part of this work was done during the 18th Annual Workshop on Probability and Combinatorics, McGill University's Bellairs Institute, Holetown, Barbados (March 22--29, 2024).


\bibliography{burning_trees_refs} 

\begin{thebibliography}{10}

\bibitem{aldous1991continuum}
David Aldous.
\newblock The continuum random tree. ii. an overview.
\newblock {\em Stochastic Analysis}, 167:23--70, 1991.

\bibitem{alon1992transmitting}
Noga Alon.
\newblock Transmitting in the $n$-dimensional cube.
\newblock {\em Discrete Applied Mathematics}, 37:9--11, 1992.

\bibitem{bastide2021improved}
Paul Bastide, Marthe Bonamy, Anthony Bonato, Pierre Charbit, Shahin Kamali,
  Th{\'e}o Pierron, and Mika{\"e}l Rabie.
\newblock Improved pyrotechnics: Closer to the burning graph conjecture.
\newblock {\em arXiv preprint arXiv:2110.10530}, 2021.

\bibitem{bessy2018bounds}
St{\'e}phane Bessy, Anthony Bonato, Jeannette Janssen, Dieter Rautenbach, and
  Elham Roshanbin.
\newblock Bounds on the burning number.
\newblock {\em Discrete Applied Mathematics}, 235:16--22, 2018.

\bibitem{bonato2014burning}
Anthony Bonato, Jeannette Janssen, and Elham Roshanbin.
\newblock Burning a graph as a model of social contagion.
\newblock In {\em Algorithms and Models for the Web Graph: 11th International
  Workshop, WAW 2014, Beijing, China, December 17-18, 2014, Proceedings 11},
  pages 13--22. Springer, 2014.

\bibitem{bonato2016burn}
Anthony Bonato, Jeannette Janssen, and Elham Roshanbin.
\newblock How to burn a graph.
\newblock {\em Internet Mathematics}, 12(1-2):85--100, 2016.

\bibitem{devroye1998branching}
Luc Devroye.
\newblock Branching processes and their applications in the analysis of tree
  structures and tree algorithms.
\newblock In {\em Probabilistic Methods for Algorithmic Discrete Mathematics},
  pages 249--314. Springer, 1998.

\bibitem{devroye2011distances}
Luc Devroye and Svante Janson.
\newblock Distances between pairs of vertices and vertical profile in
  conditioned {G}alton--{W}atson trees.
\newblock {\em Random Structures \& Algorithms}, 38(4):381--395, 2011.

\bibitem{janson2012simply}
Svante Janson.
\newblock Simply generated trees, conditioned {G}alton--{W}atson trees, random
  allocations and condensation.
\newblock {\em Probability Surveys}, 9:103, 2012.

\bibitem{janson2016asymptotic}
Svante Janson.
\newblock Asymptotic normality of fringe subtrees and additive functionals in
  conditioned {G}alton--{W}atson trees.
\newblock {\em Random Structures \& Algorithms}, 48(1):57--101, 2016.

\bibitem{land2016upper}
Max~R Land and Linyuan Lu.
\newblock An upper bound on the burning number of graphs.
\newblock In {\em Algorithms and Models for the Web Graph: 13th International
  Workshop, WAW 2016, Montreal, QC, Canada, December 14--15, 2016, Proceedings
  13}, pages 1--8. Springer, 2016.

\bibitem{lyons2017probability}
Russell Lyons and Yuval Peres.
\newblock {\em Probability on Trees and Networks}, volume~42.
\newblock Cambridge University Press, 2017.

\bibitem{meir1978altitude}
Amram Meir and John~W Moon.
\newblock On the altitude of nodes in random trees.
\newblock {\em Canadian Journal of Mathematics}, 30(5):997--1015, 1978.

\bibitem{mitsche2017burning}
Dieter Mitsche, Pawe{\l} Pra{\l}at, and Elham Roshanbin.
\newblock Burning graphs: a probabilistic perspective.
\newblock {\em Graphs and Combinatorics}, 33:449--471, 2017.

\bibitem{mitsche2018burning}
Dieter Mitsche, Pawe{\l} Pra{\l}at, and Elham Roshanbin.
\newblock Burning number of graph products.
\newblock {\em Theoretical Computer Science}, 746:124--135, 2018.

\bibitem{norin2024burning}
Sergey Norin and J{\'e}r{\'e}mie Turcotte.
\newblock The burning number conjecture holds asymptotically.
\newblock {\em Journal of Combinatorial Theory, Series B}, 168:208--235, 2024.

\end{thebibliography}

\end{document}